\documentclass[11pt,reqno]{amsart}
\usepackage{graphicx}

\usepackage{amscd,amssymb,amsmath,amsthm}
\usepackage{graphicx}
\usepackage{color}
\usepackage{cite}
\newtheorem{thm}{Theorem}
\newtheorem{defn}{Definition}
\newtheorem{lemma}{Lemma}
\newtheorem{pro}{Proposition}
\newtheorem{rk}{Remark}

\newtheorem{hyp}{Hypothesis}

\numberwithin{equation}{section} 

\def\B{\mathcal B}

\def\R{\mathbb{R}}

\begin{document}
\title[Evolution algebra and dynamical systems of bisexual population]
{Constrained evolution algebras and dynamical systems of a bisexual population}
\author{Dzhumadil'daev A., Omirov B.A., Rozikov U.A.}

\address{A.Dzhumadil'daev\\ Institute of mathematics and mathematical modeling, Kazakhstan.}
 \email {dzhuma@hotmail.com}

\address{B.\ A.\ Omirov and  U.\ A.\ Rozikov\\ Institute of mathematics,
29, Do'rmon Yo'li str., 100125, Tashkent, Uzbekistan.}
\email {omirov@mail.ru\quad  rozikovu@yandex.ru}

\begin{abstract} Consider a bisexual population such that the set of
females can be partitioned into finitely many different types
indexed by $\{1,2,\dots,n\}$ and, similarly, that the male types are
indexed by $\{1,2,\dots,\nu \}$.  Recently an evolution algebra of bisexual population was introduced
by identifying the coefficients
of inheritance of a bisexual population as the structure constants
of the algebra. In this paper we study constrained evolution algebra
of bisexual population in which type ``1''
of females and males have preference. For such algebras
sets of idempotent and absolute nilpotent elements are known.
We consider two particular cases of this algebra, giving more constrains
on the structural constants of the algebra. By the first our constrain
we obtain an $n+\nu$- dimensional algebra with a matrix of structural constants
containing only 0 and 1. In the second case we consider $n=\nu=2$ but with general constrains.
In both cases we study dynamical systems generated by the quadratic evolution operators
of corresponding constrained algebras. We find all fixed points, limit points and
some 2-periodic points of the dynamical systems.  Moreover we study several properties
of the constrained algebras connecting them to the dynamical systems. We give some biological
interpretation of our results.

\end{abstract}
\maketitle



\section{Introduction.}

Population genetics is the study of the {\it distributions or states} and {\it changes} of allele frequency
in a population, as the population is subject to the four main evolutionary processes: natural
selection, genetic drift, mutation and gene flow. It also takes into account the factors of
recombination, population subdivision and population structure.

In this paper we study of evolution algebras and dynamical systems (of states)
of sex linked populations. Depending on the matrix of structural constants these algebras
are renamed by several distinct names. For example, in mathematical genetics,
a genetic algebra is a (possibly non-associative) algebra used to model inheritance in genetics.
Some variations of these algebras are called train algebras, special train algebras,
gametic algebras, Bernstein algebras, copular algebras, zygotic algebras and baric algebras
(also called weighted algebra). The study of these algebras was started by Etherington \cite{e1}.

In applications to genetics these algebras often have a basis corresponding
to the genetically different gametes and the structure constant of the algebra
encode the probabilities of producing offspring of various types.
The laws of inheritance are then encoded as algebraic properties of the algebra.

For surveys of genetic algebras see \cite{B}, \cite{ly}, \cite{m} and \cite{w}.

In \cite{t} an evolution algebra (EA) was introduced, (denoted by $E$)
which is an algebra over a field with a basis on
which multiplication is defined by the product of distinct basis terms being zero.
This EA is commutative,
 but not necessarily associative or power-associative \cite{t}.
 Under some conditions on the matrix of structural constants
 the algebra $E$ is a baric algebra \cite{CLR}.

In \cite{ly} an EA (denoted by $\mathcal A$) associated to the
free population is introduced and using this non-associative algebra
many results are obtained in explicit form, e.g. the explicit
description of stationary quadratic operators and the explicit
solutions of a nonlinear evolutionary equation in the absence of
selection, as well as general theorems on convergence to equilibrium
in the presence of selection. Note that the algebra $\mathcal A$ is a baric algebra.

 Recently, in \cite{LR} a bisexual population was considered and an EA (denoted by $\mathcal B$)
   using inheritance coefficients of the population was introduced. This algebra is a natural
 generalization of the algebra $\mathcal A$ of free population. Moreover, the evolution algebra of a bisexual
 population $\mathcal B$ is different from the EA $E$ defined in \cite{t}. In fact, the table of
 multiplications of  $E$ {\it non-zero} multiplications are {\it only} the square
of each basis element. However in $\mathcal B$ the square
of {\it each} basis element is zero. Note also that the algebra $\mathcal B$ is never baric, but it is a dibaric
algebra.

In \cite{V} a notion of gonosomal algebra is introduced. This algebra extends the evolution algebra
of the bisexual population $\mathcal B$.
It is shown that gonosomal algebras represent algebraically a wide
variety of sex determination systems observed in bisexual populations.
Moreover, it was shown that unlike $\mathcal B$
the gonosomal algebra is not dibaric, in general.

Note that each multiplication table of an $n$-dimensional algebra $A$ ($n$ can be infinite)
over a field $K$ generates an operator $V$ from $K^n$ to itself. The properties of the algebras
can be studied by the dynamical systems generated by the operator $V$. Note that for $E$ the operator
$V$ is a linear operator \cite{t}, for other mentioned above algebras the
operator $V$ is a quadratic operator \cite{CLR},\cite{gar}, \cite{LLR}, \cite{ly}, \cite{Ro},\cite{RV}.

  In this paper we study constrained evolution algebra
of bisexual population $\mathcal B$ in which type ``1''
of females and males have preference. For such algebras
sets of idempotent and absolute nilpotent elements are known \cite{LR}.
In Section 2 we give necessary definitions. Section 3 is devoted to dynamical systems
generated by the corresponding operator $V$.
We consider two particular cases: giving hard constrains
on the structural constants of the algebra. By the first our constrain
we obtain an $n+\nu$- dimensional algebra with a matrix of structural constants
containing only 0 and 1. In the second case we consider $n=\nu=2$ but with general constrains.
In both cases we study quadratic dynamical systems. We find all fixed points, limit points and
some 2-periodic points of the dynamical systems. In Section 4 we study several properties
of the constrained algebras connecting them to the dynamical systems. In the last section we give some biological
interpretations of our results.

\section{Preliminaries}

  Assuming that the population is bisexual we suppose that the set of
females can be partitioned into finitely many different types
indexed by $\{1,2,\dots,n\}$ and, similarly, that the male types are
indexed by $\{1,2,\dots,\nu \}$ \cite{ly}. The number $n+\nu$ is called the
dimension of the population.

The following set is called simplex
$$S^{n-1}=\{x=(x_1,\dots, x_n)\in \R^n: \, x_i\geq 0, \, \sum_{i=1}^nx_i=1\}.$$

The population is described by its
state vector $(x,y)$ in $S^{n-1}\times S^{\nu-1}$. Vectors
$x$ and $y$ are the probability distributions of the females and
males over the possible types, i.e., $x\in S^{n-1}$ and $y\in S^{\nu-1}$.

Denote $S=S^{n-1}\times S^{\nu-1}$. We call the partition into types
hereditary if for each possible state $z=(x,y)\in S$ describing the
current generation, the state $z'=(x',y')\in S$ is uniquely defined
describing the next generation. This means that the association
$z\mapsto z'$ defines a map $V \colon S\to S$ called the evolution
operator.

For any point $z^{(0)}\in S$ the sequence $z^{(t)}=V(z^{(t-1)}),
t=1,2,\dots$ is called the trajectory of $z^{(0)}$.

 Let $P_{ik,j}^{(f)}$ and $P_{ik,l}^{(m)}$ be inheritance coefficients
 defined as the probability that a female offspring is type $j$ and, respectively,
 that a male offspring is of type $l$, when the parental pair is
 $ik$ $(i,j=1,\dots,n$; and $k,l=1,\dots,\nu)$. We have
\begin{equation}\label{2}
P_{ik,j}^{(f)}\geq 0, \ \ \sum_{j=1}^nP_{ik,j}^{(f)}=1; \ \
P_{ik,l}^{(m)}\geq 0, \ \ \sum_{l=1}^\nu P_{ik,l}^{(m)}=1.
\end{equation}

Let $z'=(x',y')$ be the state of the offspring population at the
birth stage. This state is obtained from inheritance coefficients as
\begin{equation}\label{3}
x'_j= \sum_{i,k=1}^{n,\nu}P_{ik,j}^{(f)}x_iy_k; \ \ y'_l=
\sum_{i,k=1}^{n,\nu} P_{ik,l}^{(m)}x_iy_k.
\end{equation}
We see from (\ref{3}) that for a bisexual population (BP) the evolution operator is a
quadratic mapping of $S$ into itself.

Following \cite{LR} we give an algebra structure on the vector space
 $\R^{n+\nu}$ which is closely related to the map defined by (\ref{3}).

 Consider $\{e_1,\dots,e_{n+\nu}\}$ the canonical basis on $\R^{n+\nu}$
 and divide the basis as $e^{(f)}_i=e_i$, $ i=1,\dots,n$ and $e^{(m)}_i=e_{n+i}$,
 $i=1,\dots,\nu$.

 Now introduce on $\R^{n+\nu}$ the table of multiplications defined by
\begin{equation}\label{4}
\begin{array}{ll}
e^{(f)}_ie^{(m)}_k = e^{(m)}_ke^{(f)}_i=\frac{1}{2} \left(\displaystyle\sum_{j=1}^nP_{ik,j}^{(f)}e^{(f)}_j+ \sum_{l=1}^{\nu}P_{ik,l}^{(m)}e^{(m)}_l\right),\\[3mm]
e^{(f)}_ie^{(f)}_j =0,\ \ i,j=1,\dots,n; \ \ e^{(m)}_ke^{(m)}_l =0,
\ \ k,l=1,\dots,\nu. \end{array} \end{equation}
Thus, we identify the coefficients of bisexual inheritance as the
structure constants of an algebra, i.e., a bilinear mapping of
$\R^{n+\nu}\times \R^{n+\nu}$ to $\R^{n+\nu}$.

The general formula for the multiplication is the extension of
(\ref{4}) by bilinearity, i.e., for $z,t\in\R^{n+\nu}$,
\[ z=(x,y)=\sum_{i=1}^nx_ie_i^{(f)}+\sum_{j=1}^\nu y_je_j^{(m)}, \ \ t=(u,v)=\sum_{i=1}^nu_ie_i^{(f)}+\sum_{j=1}^\nu v_je_j^{(m)}\]
using (\ref{4}), we obtain

\begin{align}\label{5}
 zt & = \frac{1}{2} {\displaystyle \sum_{k=1}^n\left(\sum_{i=1}^n\sum_{j=1}^\nu
P_{ij,k}^{(f)}(x_iv_j+u_iy_j)\right)}e^{(f)}_k \\ & {} +  \frac{1}{2} {\displaystyle\sum_{l=1}^\nu\left(\sum_{i=1}^n\sum_{j=1}^\nu
P_{ij,l}^{(m)}(x_iv_j+u_iy_j)\right)}e^{(m)}_l \, \notag.
\end{align}

 This algebraic interpretation is very useful. For
example, a BP state $z=(x,y)$ is an equilibrium (fixed point,
$V(z)=z$) precisely when $z$ is an idempotent element of the algebra $\R^{n+\nu}$.
The element $z$ is called an {\it absolute nilpotent} if $z^2=0$,
i.e., $z$ is zero-point of the evolution operator $V$.

The algebra ${\mathcal B}={\mathcal B}_V$ generated by the evolution
operator $V$ (see (\ref{3})) is called the {\it evolution algebra of
the bisexual population} (EABP) \cite{LR}. In \cite{LR} the basic properties of the EABP are studied.

Consider the following constrain on heredity coefficients (\ref{2}):
\begin{equation}\label{v1}P^{(f)}_{ik,j}=\left\{\begin{array}{lll}
a_{ij}, \ \ \mbox{if} \ \ k=1 \\[2mm]
1, \ \ \ \ \mbox{if} \ \ k\ne 1, j=1 \\[2mm]
0, \ \ \ \ \mbox{if} \ \ k\ne 1, j\ne 1
\end{array}\right. ,
\ \  P^{(m)}_{ik,l}=\left\{
\begin{array}{lll}
b_{kl}, \ \ \mbox{if} \ \ i=1 \\[2mm]
1, \ \ \ \ \mbox{if} \ \ i\ne 1, l=1 \\[2mm]
0, \ \ \ \ \mbox{if} \ \ i\ne 1, l\ne 1.
\end{array} \right.
\end{equation}
It is easy to see that the matrices $A=(a_{ij})$ and $B=(b_{kl})$ by (\ref{2}) satisfy the
following conditions
\begin{equation}\label{v}
 a_{ij}\geq 0, \ \ \sum_{j=1}^na_{ij}=1, \ \
i=1,\dots,n; \ \ b_{kl}\geq 0, \ \ \sum_{l=1}^\nu b_{kl}=1, \ \
k=1,\dots,\nu \, ,\end{equation} i.e. both matrices are stochastic.

The condition (\ref{v1}) has very clear biological
treatment: the type ``1'' of females and the type ``1'' of males
have preference, i.e., any type of female (male) can be born if its
father (mother) has type ``1''. If the father (mother) has type $\ne
1$ then only type ``1'' female (male) can be born.

Under condition (\ref{v1}) the multiplication (\ref{4}) became as

\begin{equation}\label{v2}
\begin{array}{ll}
e^{(f)}_ie^{(m)}_k = e^{(m)}_ke^{(f)}_i=\frac{1}{2}
\begin{cases}
\displaystyle \sum_{j=1}^na_{1j}e^{(f)}_j+
\sum_{l=1}^{\nu}b_{1l}e^{(m)}_l & \mbox{if} \ \ i=1,k=1\\[2mm]
\displaystyle\sum_{j=1}^na_{ij}e^{(f)}_j+ e^{(m)}_1 & \mbox{if} \ \ i\ne 1, k=1\\[2mm]
\displaystyle e^{(f)}_1+\sum_{l=1}^\nu b_{kl}e^{(m)}_l & \mbox{if} \ \ i =1, k\ne 1\\[4mm]
e^{(f)}_1+e^{(m)}_1 & \mbox{if} \ \ i\ne 1, k\ne 1 \, ,
\end{cases}\\[9mm]
e^{(f)}_ie^{(f)}_j =0,\ \ i,j=1,\dots,n; \ \qquad  e^{(m)}_ke^{(m)}_l =0,
\ \ k,l=1,\dots,\nu \, .
\end{array}
\end{equation}
The operator defined by (\ref{3}) has the following form
\begin{equation}\label{v3}\begin{array}{llll}
x'_1=\displaystyle\sum_{i=1}^n\left(a_{i1}y_1+\sum_{k=2}^\nu y_k\right)x_i\\[2mm]
x'_j=y_1\displaystyle\sum_{i=1}^na_{ij}x_i, \ \ j\ne 1\\[4mm]
y'_1=\displaystyle\sum_{k=1}^\nu\left(b_{k1}x_1+\sum_{i=2}^n x_i\right)y_k\\[2mm]
y'_l=x_1\displaystyle\sum_{k=1}^\nu b_{kl}y_k, \ \ l\ne 1 \, .
\end{array}
\end{equation}
Denote by $\B_1$ the EABP defined by the multiplication table
(\ref{v2}). This algebra was introduced in \cite{LR} and the sets of idempotent and
absolute nilpotent elements of the algebra $\B_1$ are described.
In this paper we continue the study of this algebra and corresponding dynamical systems.

\section{Dynamical systems generated by evolution operator (\ref{v3})}

Note that (\ref{v3}) maps $S$ to itself. Let $m\geq 1$ be a natural number.
If we write $z^{(m)}$ for the power $(\cdots(z^2)^2\cdots)$ ($m$
times) with $z^{(0)}\equiv z$ then the trajectory (or dynamical system (DS)) with initial state
$z$ is $V^m(z)=z^{(m)}$.
In this section we shall study
the behavior of the trajectory for each given initial point $z\in S$.
But this problem is a rather difficult in general, because the dynamical system generated by operator (\ref{v3}) is a complicated
non linear system. Therefore,  we shall consider two particular cases of the system.

\subsection{A simple case: a hard constrain.} Let us start with a particular
case of $\B_1$ and completely study the corresponding dynamical system.

Consider the case
\begin{equation}\label{v7}
a_{ij}=
\begin{cases}
1 & \mbox{if} \ \  i=j \\
0  & \mbox{if} \ \   i\ne j
\end{cases} \  \ , \ \  b_{kl}= \begin{cases}
1 & \mbox{if} \ \ k=l\\
0 & \mbox{if} \ \  k\ne l \, ,
\end{cases}
\end{equation}
then the operator defined by (\ref{v3}) has the following form
\begin{equation}\label{v3a} H:\left\{\begin{array}{llll}
x'_1=1-y_1(1-x_1)\\[2mm]
x'_j=y_1x_j, \ \ j=2,\dots,n\\[2mm]
y'_1=1-x_1(1-y_1)\\[2mm]
y'_l=x_1y_l, \ \ l=2,\dots, \nu.
\end{array}\right.
\end{equation}
Let Fix$(H)$ be the set of fixed points of $H$, i.e., $H(z)=z$, $z=(x,y)\in S$. By \cite[Proposition 10.6]{LR} we get
\begin{pro}\label{p1} The set of fixed points is
$$\begin{array}{ll}
{\rm Fix}(H)=\left\{((1,0,\dots,0),(y_1,\dots,y_\nu)):\displaystyle\sum_{k=1}^\nu y_k=1\right\}\cup\\[4mm]
\ \ \ \ \ \ \ \ \ \ \ \ \ \left\{((x_1,\dots,x_n),(1,0,\dots,0)): \displaystyle\sum_{i=1}^n x_i=1\right\}.
\end{array}
$$
\end{pro}

Now we shall examine the type of the fixed points.

\begin{defn} (see \cite{D}). A fixed point $z$ of an operator $H$ is called hyperbolic if
its Jacobian $J$ at $z$ has no eigenvalues on the unit circle.
\end{defn}

\begin{defn} (see \cite{D}). A hyperbolic fixed point $z$ is called:
\begin{itemize}
\item attracting if all the eigenvalues of the Jacobi matrix $J(z)$ are less than 1 in
absolute value;
\item repelling if all the eigenvalues of the Jacobi matrix $J(z)$ are greater than 1 in
absolute value;
\item a saddle otherwise.
\end{itemize}
\end{defn}

To find the type of a fixed point of the operator (\ref{v3a}) we write
the Jacobi matrix, replacing $x_1$ by $1-\displaystyle\sum_{i=2}^nx_i$ and $y_1$
by $1-\displaystyle\sum_{j=2}^\nu y_j$, which is $(n-1)\times(\nu-1)$-matrix of the form:

$$J(s)=J_H(s)=\left(\begin{array}{cccccccc}
y_1&0&\dots&0&-x_2&-x_2&\dots&-x_2\\[3mm]
0&y_1&\dots&0&-x_3&-x_3&\dots&-x_3\\[3mm]
\vdots&\vdots&\dots&\vdots &\vdots &\vdots&\dots&\vdots\\[3mm]
0&0&\dots&y_1&-x_n&-x_n&\dots&-x_n\\[3mm]
-y_2&-y_2&\dots&-y_2&x_1&0&\dots&0\\[3mm]
-y_3&-y_3&\dots&-y_3&0&x_1&\dots&0\\[3mm]
\vdots&\vdots&\dots&\vdots &\vdots &\vdots&\dots&\vdots\\[3mm]
-y_\nu&-y_\nu &\dots&-y_\nu &0&0&\dots&x_1
\end{array}\right).$$
It is easy to see that $J(((1,0,\dots,0),(y_1,\dots,y_\nu)))$ has eigenvalues equal to 1, $y_1$, and $J(((x_1,\dots, x_n),(1,0,\dots,0)))$ has eigenvalues 1 and $x_1$ therefore
all fixed points are not hyperbolic.

The following theorem completely describes the limit points of trajectories for operator $H$.
\begin{thm}\label{ts} Let $z^{(0)}=((x^{(0)}_1,\dots,x_n^{(0)}),(y^{(0)}_1,\dots,y_\nu^{(0)}))\in S$ be an initial point
\begin{itemize}
\item[(i)] If $x_1^{(0)}\ne 0$ and $y_1^{(0)}\ne 0$
then
$$\lim_{m\to +\infty}z^{(m)}=\left\{\begin{array}{lll}
((1,0,\dots,0),(1,0,\dots,0)), \ \ \ \ \ \ \ \ \mbox{if} \ \ \ \ \ \ x_1^{(0)}=y_1^{(0)}\\[3mm]
((1,0,\dots,0),(1-x_1^{(0)}+y_1^{(0)},y^*_2,\dots,y^*_\nu)), \ \ \mbox{if} \ \ x_1^{(0)}>y_1^{(0)}\\[3mm]
((1+x_1^{(0)}-y_1^{(0)},x^*_2,\dots, x^*_n),(1,0,\dots,0)), \ \ \mbox{if} \ \ x_1^{(0)}<y_1^{(0)},
\end{array}\right.
$$
where
$$x_j^*={(y_1^{(0)}-x_1^{(0)})x^{(0)}_j\over 1-x^{(0)}_1}, \ \ j=2,\dots,n; \ \ y_l^*={(x_1^{(0)}-y_1^{(0)})y^{(0)}_l\over 1-y_1^{(0)}}, \ \ l=2,\dots,\nu.$$
\item[(ii)] If $x^{(0)}_1=0$ then
$$z^{(m)}=H(z^{(0)})=\left((1-y^{(0)}_1, y^{(0)}_1x_2^{(0)},\dots, y^{(0)}_1x_n^{(0)}), (1,0,\dots,0)\right), \ \ \mbox{for all} \ \ m\geq 1.$$

\item[(iii)] If $y^{(0)}_1=0$ then
$$z^{(m)}=H(z^{(0)})=\left((1,0,\dots,0),(1-x^{(0)}_1, x^{(0)}_1y_2^{(0)},\dots, x^{(0)}_1y_\nu^{(0)})\right), \ \ \mbox{for all} \ \ m\geq 1.$$
\end{itemize}
\end{thm}
\begin{proof} From (\ref{v3a}) we get
\begin{equation}\label{v3b} H^{m+1}(z^{(0)})\,:\,\left\{\begin{array}{llll}
x^{(m+1)}_1=1-y^{(m)}_1(1-x^{(m)}_1)\\[2mm]
x^{(m+1)}_j=y^{(m)}_1x^{(m)}_j, \ \ j=2,\dots,n\\[2mm]
y^{(m+1)}_1=1-x^{(m)}_1(1-y^{(m)}_1)\\[2mm]
y^{(m+1)}_l=x^{(m)}_1y^{(m)}_l, \ \ l=2,\dots, \nu.
\end{array}\right.
\end{equation}

Since $y_1^{(m)}\in [0,1]$ we have
$$x^{(m+1)}_1=1-y^{(m)}_1(1-x^{(m)}_1)\geq 1-(1-x^{(m)}_1)=x_1^{(m)},$$
$$ x^{(m+1)}_j\leq x_j^{(m)}, \ \ j=2,\dots,n.$$
Thus $x^{(m)}_1$ is a non-decreasing sequence, which bounded from above by 1 and for each $j=2,\dots,n$
the sequence $x^{(m)}_j$ is non-increasing and with lower bound $0$. Consequently, each $x^{(m)}_j$ has a limit say $\alpha_j$, $j=1,\dots,n$.  Similarly one shows that
$y^{(m)}_l$ also has a limit, say $\beta_l$, $l=1,\dots,\nu$.
Hence, the limit $\displaystyle\lim_{m\to\infty}z^{(m)}=z_*$ exists. It is clear that $z_*\in$Fix$(H)$,
and  $z_*=z_*(z^{(0)})$, i.e.,
it depends on the initial point $z^{(0)}$. Now we shall find $z_*(z^{(0)})$.
Subtracting from the first equality of (\ref{v3b}) the third one we get
\begin{equation}\label{13e}
x_1^{(m+1)}-y_1^{(m+1)}=x_1^{(m)}-y_1^{(m)}, \ \ m=0,1,2,\dots
\end{equation}
Iterating these equalities we obtain
\begin{equation}\label{13e1}
x_1^{(m)}=y_1^{(m)}+(x_1^{(0)}-y_1^{(0)}), \ \ m=1,2,\dots
\end{equation}
{\sl Proof of} (i).
From the first and third equations of (\ref{v3b}), and (\ref{13e1}) for limit values $\alpha_1$ and $\beta_1$ we get the following equations
\begin{equation}\label{vab} \left\{\begin{array}{lll}
\alpha_1=1-\beta_1(1-\alpha_1)\\[2mm]
\beta_1=1-\alpha_1(1-\beta_1)\\[2mm]
\alpha_1=\beta_1+(x_1^{(0)}-y_1^{(0)}).
\end{array}\right.
\end{equation}
It is easy to see that the system (\ref{vab}) has following solution:
\begin{equation}\label{sol}
(\alpha_1,\beta_1)=\left\{\begin{array}{lll}
(1,1), \ \ \mbox{if} \ \ x_1^{(0)}=y_1^{(0)}\\[3mm]
(1, 1-x_1^{(0)}+y_1^{(0)}), \ \ \mbox{if} \ \ x_1^{(0)}>y_1^{(0)}\\[3mm]
(1+x_1^{(0)}-y_1^{(0)},1), \ \ \mbox{if} \ \ x_1^{(0)}<y_1^{(0)}.
\end{array}\right.
\end{equation}
Thus if $x_1^{(0)}=y_1^{(0)}$ then $z_*=z_*(z^{(0)})=((1,0,\dots,0),(1,0,\dots,0))$.
Consider now two cases:

{\it Case $x_1^{(0)}>y_1^{(0)}$}: In this case $y_1^{(0)}<1$. By (\ref{sol}) we have
\begin{equation}\label{lims}
\displaystyle\lim_{m\to\infty}x^{(m)}_j=0, \ \ j=2,\dots,n; \ \ \lim_{m\to\infty}y^{(m)}_1=1-x_1^{(0)}+y_1^{(0)}.
\end{equation}

Now we shall calculate $\displaystyle\lim_{m\to\infty}y^{(m)}_l$, $l=2,\dots,\nu$.
From the third and last equalities of (\ref{v3b}) we get
\begin{equation}\label{1q}
1-y^{(m+1)}_1=x^{(m)}_1(1-y^{(m)}_1)=x^{(m)}_1x^{(m-1)}_1(1-y^{(m-1)}_1)=\dots =(1-y_1^{(0)})\displaystyle\prod_{k=0}^mx^{(k)}_1.
\end{equation}
\begin{equation}\label{2q}
y^{(m+1)}_l=x^{(m)}_1y^{(m)}_l=x^{(m)}_1x^{(m-1)}_1y^{(m-1)}_l=\dots=y_l^{(0)}\displaystyle\prod_{k=0}^mx^{(k)}_1, \ \ l=2,\dots, \nu.
\end{equation}
Taking limit from both side of (\ref{1q}) and using (\ref{lims}) we obtain
$$\displaystyle\lim_{m\to\infty}\prod_{k=0}^mx^{(k)}_1={x_1^{(0)}-y_1^{(0)}\over 1-y_1^{(0)}}.$$
Using this equality from (\ref{2q}) we get
$$\displaystyle\lim_{m\to\infty}y_l^{(m)}={(x_1^{(0)}-y_1^{(0)})y^{(0)}_l\over 1-y_1^{(0)}}, \ \ l=2,\dots,\nu.$$

{\it Case $x_1^{(0)}<y_1^{(0)}$}: This case is similar to the previous case and one obtains
$$\displaystyle\lim_{m\to\infty}x_j^{(m)}={(y_1^{(0)}-x_1^{(0)})x^{(0)}_j\over 1-x_1^{(0)}}, \ \ j=2,\dots,n.$$
These equalities complete the proof of part (i).

{\sl Proof of} (ii) Let $x^{(0)}_1=0$, $y_1^{(0)}\ne 0$. In this case by (\ref{v3b}) we get
\begin{equation}\label{v3c} H(z^{(0)})\,:\,\left\{\begin{array}{llll}
x^{(1)}_1=1-y^{(0)}_1\\[2mm]
x^{(1)}_j=y^{(0)}_1x^{(0)}_j, \ \ j=2,\dots,n\\[2mm]
y^{(1)}_1=1\\[2mm]
y^{(1)}_l=0, \ \ l=2,\dots, \nu.
\end{array}\right.
\end{equation}
Thus by Proposition \ref{p1} we have $H(z^{(0)})\in$Fix$(H)$. This completes the proof of (ii).

Proof of (iii) is similar to the case (ii).\end{proof}

\subsection{All possible constrains for $n=\nu=2$.}
 Consider the case $n=\nu=2$. In this case since
 $$x_1+x_2=y_1+y_2=a_{11}+a_{12}=a_{21}+a_{22}=b_{11}+b_{12}=b_{21}+b_{22}=1$$
 denoting
 \begin{equation}\label{abc}
 a=a_{12}, \ \ b=a_{22}, \ \ c=b_{12},\ \ d=b_{22}, \ \ x=x_2, \ \ y=y_2
 \end{equation}
 the operator (\ref{v3}) can be reduced to the following operator
 \begin{equation}\label{22}
 T:\left\{\begin{array}{ll}
 x'=(1-y)(a+(b-a)x)\\[2mm]
 y'=(1-x)(c+(d-c)y),
 \end{array}\right.
 \end{equation}
 where $a,b,c,d,x,y\in [0,1]$.
 Thus we get a quadratic operator with four independent parameters.
 In this subsection we shall study the dynamical system
 generated by this operator.

 \subsubsection{Fixed points of (\ref{22})}
 To find fixed points of (\ref{22}) we should solve the following
\begin{equation}\label{2a}
 \begin{array}{ll}
 x=(1-y)(a+(b-a)x)\\[2mm]
 y=(1-x)(c+(d-c)y).
 \end{array}
 \end{equation}

Denote by Fix$(T)$ the set of all fixed points of the operator $T$ given by (\ref{22}).

From the first equation of the system we find
\begin{equation}\label{x1}
(1-(b-a)(1-y))x=a(1-y).
\end{equation}

{\it Case:} $(1-y)(b-a)=1$. In this case from (\ref{x1}) we get $a=0$ or $y=1$.
If $a=0$ then $(1-y)b=1$, i.e., $y=1-1/b$. This gives non-negative $y$ only if $b=1$, i.e. in this case $y=0$.
Under these assumptions from the second equation of the system (\ref{2a}) we get $c=0$ or $x=1$. Assume $c=0$ then it is easy
to see that $(x,0)$ is a solution of the system for any $x\in [0,1]$. In case $c\ne 0$ we get solution $(1,0)$.

{\it Case:} $(1-y)(b-a)\ne 1$. From (\ref{x1}) we obtain
\begin{equation}\label{x2}
x={a(1-y)\over a(1-y)+1-b(1-y)}.
\end{equation}
From this equality it is clear that for each $y\in [0,1]$ the corresponding $x$ also will be in $[0,1]$.
Substituting (\ref{x2}) in the second equation of (\ref{2a}) we get
\begin{equation}\label{y1}
[b(d-c)+a-b]y^2+[2bc-bd-a+b-c+d-1]y+c(1-b)=0.
\end{equation}

{\sl Subcase:} $b(d-c)=b-a$ and $c(b-1)=1-d$. Then $c(1-b)=0$. Under these conditions we get
$a=c=0$, $d=1$ and $b\ne 1$ which gives fixed points $(0,y)$ for any $y\in [0,1]$.
Moreover if $a=c\ne 0$, $b=d=1$, then the following points are fixed
$$P(y)=\left({a(1-y)\over a(1-y)+y},y\right), \ \ \forall y\in [0,1].$$

{\sl Subcase:} $b(d-c)=b-a$ and $c(b-1)\ne 1-d$.
Then the following point is a fixed point
$$P_0=\left({a(1-d)\over a(1-d)+c(1-b)+(1-b)(1-d)}, {c(1-b)\over c(1-b)+(1-d)}\right).$$

{\sl Subcase:} $b(d-c)<b-a$. In this case we have
$$D= [2bc-bd-a+b-c+d-1]^2-4[b(d-c)-(b-a)]c(1-b)>0.$$
Thus equation (\ref{y1}) has two solutions $y_1<y_2$.  By Vieta's theorem one can see that $y_1<0$.
For $y_2$ we have
$$y_2={(2bc-bd-a+b-c+d-1)+\sqrt{D}\over 2(b(c-d)+b-a)}.$$

{\sl Subcase:} $b(d-c)>b-a$. In this case we assume
$$D= [2bc-bd-a+b-c+d-1]^2-4[b(d-c)-(b-a)]c(1-b)>0.$$
Thus equation (\ref{y1}) has two solutions $0<y_1<y_2$, with

$$y_{1,2}={-(2bc-bd-a+b-c+d-1)\pm \sqrt{D}\over 2(b(d-c)+a-b)}.$$
Choose parameters $a,b,c,d$ such that $y_1$ or/and $y_2$ belong in $[0,1]$ then corresponding
value of $x$ is defined by (\ref{x2}).
Denote this fixed point by $Q_i=(x_i,y_i)$, $i=1,2$.

Summarizing we get the following

\begin{pro}\label{p2} The set {\rm Fix}$(T)$ has the following form

$${\rm Fix}(T)=\left\{\begin{array}{llllllll}
\{(0,0)\}, \, \mbox{if} \ \ a=c=0, b\ne 1, d\ne 1,\\[2mm]
\{(x,0), \, x\in[0,1]\}, \ \ \mbox{if} \ \ a=c=0, b=1,\\[2mm]
\{(0,y), \, y\in[0,1]\}, \ \ \mbox{if} \ \ a=c=0, d=1,\\[2mm]
\{(1,0)\}, \ \ \mbox{if} \ \ a=0, b=1, c\ne 0,\\[2mm]
\{P(y), \, y\in[0,1]\}, \ \ \mbox{if} \ \ a=c\ne 0, b=d=1,\\[2mm]
\{P_0\}, \ \ \ \mbox{if} \ \ b(d-c)=b-a, c(b-1)\ne 1-d,\\[2mm]
\{Q_2\}, \ \ \mbox{if} \ \ b(d-c)<b-a,\\[2mm]
\{Q_1,Q_2\}, \ \ \ \mbox{otherwise}
\end{array}
\right.
$$
\end{pro}
\begin{rk}
Note that the case $a=c=0$, $b=d=1$ is a particular case of operator (\ref{v3a}). Therefore in the sequel of this section we consider the case
when $a+c\ne 0$ or $b+d\ne 2$. Thus we will consider the cases where the fixed points $P(y)$, $P_0$, $Q_i$, $i=1,2$ exist.
\end{rk}
\subsubsection{Type of fixed points}
To study type of fixed points consider Jacobi matrix of the operator (\ref{22}) at a fixed point $z=(x,y)$:
$$J(z)=J_T(z)=\left(\begin{array}{cc}
(b-a)(1-y)&-(a+(b-a)x)\\[3mm]
-(c+(d-c)y)&(d-c)(1-x)
\end{array}\right).$$

For $P(y)$ we have $a=c\ne 0$ and $b=d=1$. In this case one can find two eigenvalues: $\lambda_1=1$ and
$$\lambda_2=\lambda_2(a,y)={(1-a)y-[a+(1-a)y]^2\over (1-a)y+a}.$$
Now we shall prove that $|\lambda_2|<1$ for any $a\in(0,1)$ and $y\in [0,1]$.
Denote $t=a+(1-a)y$. We have $t\in [a,1]$. The function $\lambda_2$ can be written as
$$\lambda_2=\lambda_2(t)=1-t-{a\over t}.$$ It is easy to see that
$$-a=\lambda_2(a)=\lambda_2(1)=\min_{t\in [a,1]}\lambda_2(t)\leq \lambda_2(t)\leq \max_{t\in [a,1]}\lambda_2(t)=\lambda_2(\sqrt{a})=1-2\sqrt{a}.$$
From this inequalities we get
$$|\lambda_2|\leq \max\{a, |1-2\sqrt{a}|\}<1, \ \ \forall a\in (0,1).$$
Thus the fixed point $P(y)$ is not hyperbolic.

At the fixed point $P_0$ the eigenvalues of $J(P_0)$ has very long expression.
Using Maple and varying the parameters $a,b,c,d$ (taking into account the conditions $b(d-c)=b-a$ and $c(b-1)\ne 1-d$)
 one can check that $P_0$ is attractive. For example,
\begin{itemize}
\item[-] if $a=0.18$, $b=0.2$, $c=0.3$, $d=0.4$ then eigenvalues are approximately $-0.197$ and $0.298$.
\item[-] if $a=0$, $b=0$, $c=0.9$, $d=0.1$ then eigenvalues are $-0.8$ and $0$.
\item[-] if $a=0.81$, $b=0.9$, $c=0.7$, $d=0.8$ then eigenvalues are   $-0.743$, $0.845$.
\end{itemize}
Moreover, giving values of three parameters and plot the eigenvalues as functions of the remaining parameter,
one can see that the absolute value of the functions are less than 1. Thus based on numerical computations
we conjecture that $P_0$ is attractive for any parameters $a,b,c,d$ with conditions $b(d-c)=b-a$ and $c(b-1)\ne 1-d$.

Checking the type of $Q_1$ and $Q_2$ is also very difficult, because corresponding
eigenvalues have very long expression depending on four parameters.
Using Maple one can get the following results for $Q_2$ in case $b(d-c)<b-a$:
\begin{itemize}
\item[-] if $a=0.1$, $b=0.9$, $c=0.2$, $d=0.3$ then eigenvalues are approximately $-0.0205$, $0.7662$.
\item[-] if $a=0.7$, $b=0.9$, $c=0.2$, $d=0.3$ then eigenvalues are approximately $-0.3245$, $0.5340$.
\item[-] if $a=0.7$, $b=0.71$, $c=0.2$, $d=0$ then eigenvalues are   $-0.3949$, $0.3367$.
\end{itemize}
Thus we also have conjecture that $Q_2$ is an attracting point for any $a,b,c,d$ with  $b(d-c)<b-a$.
Similar computations can be done for $Q_1$.

\subsubsection{Dynamics for $a=c$, $b=d$.}
In this case the operator (\ref{22}) has the following form
\begin{equation}\label{20}
 T:\left\{\begin{array}{ll}
 x'=(1-y)(a+(b-a)x)\\[2mm]
 y'=(1-x)(a+(b-a)y),
 \end{array}\right.
 \end{equation}

The following lemma gives full description of fixed points.

\begin{lemma}\label{l1} The set of fixed points {\rm Fix}$(T)$ of the operator (\ref{20}) has the following form
\begin{itemize}
\item[a.] If $b=1$ then
$${\rm Fix}(T)=\left\{\begin{array}{ll}
\{(1,0)\}\cup\{(0,y), \forall y\in [0,1]\}, \ \ \mbox{if} \ \ a=0\\[2mm]
\left\{\left(x,{a(1-x)\over a(1-x)+a}\right), \forall x\in [0,1]\right\}, \ \ \mbox{if} \ \ a\ne 0.
\end{array}\right.
$$
\item[b.] If $b\ne 1$ then there is a unique fixed point, i.e.,
 $${\rm Fix}(T)=\left\{\begin{array}{ll}
 \{({a\over 1+a}, {a\over 1+a})\}, \ \ \mbox{if} \ \ a=b\\[3mm]
 \{(x_*,x_*)\}, \ \ \mbox{if} \ \ a\ne b
 \end{array}\right.
 $$where
  $$x_*={1-b+2a-\sqrt{(1-b)^2+4a}\over 2(a-b)}.$$
 \end{itemize}
\end{lemma}
\begin{proof}
Follows from detailed (but simple) analysis of the system
\begin{equation}\label{20f}
 \left\{\begin{array}{ll}
 x=(1-y)(a+(b-a)x)\\[2mm]
 y=(1-x)(a+(b-a)y).
 \end{array}\right.
 \end{equation}
For example, subtracting from the first equation of the system the second one we get $x-y=b(x-y)$, i.e., $x=y$ if $b\ne 1$.
Consequently in the case $b\ne 1$ the solutions are given by solution of $x=(1-x)(a-(b-a)x)$.
\end{proof}
\begin{lemma}\label{l2} Type of fixed points mentioned in Lemma \ref{l1} as the following
\begin{itemize}
\item[(a)] If $b=1$ the fixed points mentioned in part a. of Lemma \ref{l1} are not hyperbolic.

\item[(b)] If $b\in [0,1)$ then for fixed points mentioned in part b. we have
 $$({a\over 1+a}, {a\over 1+a})\ \ \mbox{is {\rm attractive}}$$ and
 $$(x_*,x_*)=\left\{\begin{array}{lll}
{\rm attractive}, \ \ \mbox{if} \ \ 0\leq a<1-{(1-b)^2\over 4}\\[3mm]
{\rm non hyperbolic}, \ \ \ \ \ \mbox{if} \ \  a=1-{(1-b)^2\over 4}\\[3mm]
{\rm saddle}, \ \ \mbox{if} \ \ 1-{(1-b)^2\over 4}<a\leq 1,\\[3mm]
 \end{array}\right.$$
\end{itemize}
\end{lemma}
\begin{proof} (a) Simple computations show that for each fixed point mentioned in part a. of Lemma \ref{l1} one of eigenvalues of the corresponding Jacobian is equal to 1 and the second eigenvalue is less than 1.

(b) For the case $a=b<1$ the eigenvalues are $-a$ and $a$, therefore $({a\over 1+a}, {a\over 1+a})$ is attractive.
For the case $a\ne b$ we have the eigenvalues $\lambda_1=b<1$ and
$\lambda_2=1-\sqrt{(1-b)^2+4a}$. It is easy to see that
$$|\lambda_2|=\left\{\begin{array}{lll}
<1, \ \ \mbox{if} \ \ 0\leq a<1-{(1-b)^2\over 4}\\[3mm]
1, \ \ \ \ \ \mbox{if} \ \  a=1-{(1-b)^2\over 4}\\[3mm]
>1, \ \ \mbox{if} \ \ 1-{(1-b)^2\over 4}<a\leq 1,\\[3mm]
 \end{array}\right.
$$
This completes the proof.
\end{proof}

{\it Limit points for} $a=b$. One very simple but interesting case is $a=b$. In this case we have
$$x^{(n+1)}=a(1-y^{(n)})=a-a^2+a^2x^{(n-1)}=a-a^2+a^3-a^3y^{(n-2)}=\dots=$$
$${a(1-(-a)^{n+1})\over 1+a}+(-a)^{n+1}x^{(0)} \ \ ({\rm or} \,  y^{(0)} \, {\rm depending \, on \, parity \, of}\, n).$$
Similarly for $y^{(n+1)}$ we obtain
$$y^{(n+1)}={a(1-(-a)^{n+1})\over 1+a}+(-a)^{n+1}y^{(0)} \ \ ({\rm or} \,  x^{(0)} \, {\rm depending \, on \, parity \, of}\, n).$$
Consequently, we have the following
$$\lim_{n\to\infty}(x^{(n)},y^{(n)})=({a\over 1+a}, {a\over 1+a}),\ \ \mbox{if} \ \ 0\leq a=b<1.$$
Moreover if $a=b=1$ then for any initial point $(x^{(0)},y^{(0)})$ its trajectory $(x^{(n)},y^{(n)})$ is a 2-periodic sequence:
$$(x^{(n)},y^{(n)})=\left\{\begin{array}{ll}
(x^{(0)},y^{(0)}),  \ \ \mbox{if} \ \ n \ \ \mbox{is even}\\[2mm]
(1-y^{(0)},1-x^{(0)}),  \ \ \mbox{if} \ \ n \ \ \mbox{is odd}\\[2mm]
\end{array}\right.
$$

{\it Limit points for} $a\ne b$. For $(x^{(0)},y^{(0)})\in [0,1]^2$, let $(x^{(n)},y^{(n)})$ be the trajectory generated by
actions of operator (\ref{20}).

The following lemma is useful
\begin{lemma} The following equalities hold
\begin{itemize}
\item[(1)] $x^{(n+1)}-y^{(n+1)}=b(x^{(n)}-y^{(n)})$ for any $n=0,1,2,\dots$.

\item[(2)] If $b=1$ then $x^{(n)}-y^{(n)}=x^{(0)}-y^{(0)}$ for any $n=1,2,\dots$.

\item[(3)] If $b\ne 1$ then
$$\lim_{n\to\infty}\left(x^{(n)}-y^{(n)}\right)=0.$$
\end{itemize}
\end{lemma}
\begin{proof} From (\ref{20}) we have
\begin{equation}\label{21}
 \left\{\begin{array}{ll}
 x^{(n+1)}=(1-y^{(n)})(a+(b-a)x^{(n)})\\[2mm]
 y^{(n+1)}=(1-x^{(n)})(a+(b-a)y^{(n)}).
 \end{array}\right.
 \end{equation}
 Subtracting from the first equation the second one we get (1). The assertions (2) and (3) follow from (1).
\end{proof}

It is easy to see that the set $I=\{(x,y)\in [0,1]^2: x=y\}$ is invariant with respect to
operator (\ref{20}), i.e., $T(I)\subset I$. Let us study the dynamical system on $I$.

Reducing the operator $T$ on the set $I$ we get
$$x'=f(x)=(1-x)(a+(b-a)x).$$
This function has the following properties:
$f(0)=a$, $f(1)=0$ and
 if $2a\geq b$ then $f(x)$ is a decreasing function, so maximal value of $f$ is $a$.
 If $2a<b$ then $f(x)$ is increasing for $x\in [0,{b-2a\over 2(b-a)}]$ and decreasing
 for $x\in [{b-2a\over 2(b-a)},1]$, then maximal value of $f$ is ${b^2\over 4(b-a)}$, which is less than 1 for any $b>2a$.
Thus $f:[0,1]\to [0,1]$.

For any $a\ne b$ the function $f(x)$ has a unique fixed point $x_* \in [0,1]$
which is defined in Lemma \ref{l1}.
It is known that $x_*\in [0,1]$ is attractive if $|f'(x_*)|<1$,  saddle if $|f'(x_*)|=1$ and repeller if $|f'(x_*)|>1$.
Solving these inequalities we get
$$x_*=\left\{\begin{array}{lll}
{\rm attractive}, \ \ \mbox{if} \ \ 0\leq a<1-{(1-b)^2\over 4}\\[3mm]
{\rm saddle}, \ \ \ \ \ \mbox{if} \ \  a=1-{(1-b)^2\over 4}\\[3mm]
{\rm repeller}, \ \ \mbox{if} \ \ 1-{(1-b)^2\over 4}<a\leq 1,\\[3mm]
 \end{array}\right.
$$
where $b\in [0,1)$.

We will apply the following lemma (see \cite[p.70]{Kes}).

\begin{lemma}\label{mrsl1}  Let $f : [0, 1]\to [0, 1]$ be a continuous function with a fixed
point $x_*\in (0, 1)$. Assume that $f$ is differentiable at $x_*$ and that $f'(x_*) <-1$. Then there exist
$p_1$, $p_2$, $0\leq p_1 < x_* < p_2\leq 1$, such that $p_1=f(p_2)$ and $p_2=f(p_1)$.
\end{lemma}

We note that the condition $f'(x_*) <-1$ of Lemma \ref{mrsl1} is satisfied iff
$1-{(1-b)^2\over 4}<a\leq 1$, i.e., when the fixed point is repeller,  then by Lemma \ref{mrsl1} it follows that
there are two 2-periodic points, denoted by $p_1$ and $p_2$, which are solutions
of the system $p_1=f(p_2), \, p_2=f(p_1)$.
Thus $p_1$ and $p_2$ are solutions of $f^2(x)=x$ which are different from the fixed point $x_*$.
Hence to find 2-periodic points one has to solve the following equation
$${f^2(x)-x\over f(x)-x}=0.$$
This equation has solutions
$$p_1={2a-b-1-\sqrt{(1-b)^2-4(1-a)}\over 2(a-b)}, \ \ p_2={2a-b-1+\sqrt{(1-b)^2-4(1-a)}\over 2(a-b)}.$$

Recall that a periodic point $p$ with of period $n$
is called an attracting periodic point if $|(f^n)'(p)|<1$.

\begin{lemma}\label{lp}
Both 2-periodic points $p_1$ and $p_2$  are attracting.
\end{lemma}
\begin{proof} Since $(f^2)'(x)=f'(f(x))f'(x)$ we have
 $$(f^2)'(p_1)=f'(f(p_1))f'(p_1)=f'(p_2)f'(p_1)=(f^2)'(p_2)=4-4a+2b-b^2.$$
 Using condition $1-{(1-b)^2\over 4}<a\leq 1$, $b\ne 1$ one can check that
 $|4-4a+2b-b^2|<1$.
\end{proof}

Summarizing we get
\begin{thm}\label{t2} For an initial point $(x^{(0)},y^{(0)})\in [0,1]^2$ its trajectory $(x^{(n)},y^{(n)})$
 generated by the operator (\ref{20}) has the following properties
\begin{itemize}
\item[1.]  $$\lim_{n\to\infty}(x^{(n)},y^{(n)})=({a\over 1+a}, {a\over 1+a}),\ \ \mbox{if} \ \ 0\leq a=b<1.$$
\item[2.]  If $a=b=1$ then
$$(x^{(n)},y^{(n)})=\left\{\begin{array}{ll}
(x^{(0)},y^{(0)}),  \ \ \mbox{if} \ \ n \ \ \mbox{is even}\\[2mm]
(1-y^{(0)},1-x^{(0)}),  \ \ \mbox{if} \ \ n \ \ \mbox{is odd}\\[2mm]
\end{array}\right.$$
\item[3.] Let $b\ne 1$ and $0\leq a<1-{(1-b)^2\over 4}$. Then there is an open neighborhood $U$ of $(x_*,x_*)$ such
that if $(x^{(0)},y^{(0)}) \in U$, then
$$\lim_{n\to\infty}(x^{(n)},y^{(n)})=(x_*,x_*).$$

\item[4.] Let $b\ne 1$ and  $1-{(1-b)^2\over 4}<a\leq 1$. Then
there is an open neighborhood $V$ of $(x_*,x_*)$ such that, if $(x^{(0)},y^{(0)})\in V\setminus\{(x_*,x_*)\}$, with $x^{(0)}=y^{(0)}$ then there exists
$k>0$ such that $(x^{(k)},y^{(k)})\notin V$. Moreover, there is a curve $\gamma\subset [0,1]^2$  through $(x_*,x_*)$ which is
invariant with respect to $T$ and for any initial point $(x^{(0)},y^{(0)})\in \gamma$  all trajectories tend to $(x_*,x_*)$.

\item[5.] If $1-{(1-b)^2\over 4}<a\leq 1$ and $x^{(0)}=y^{(0)}<x_*$ (resp.  $x^{(0)}=y^{(0)}>x_*$) then there
is an open neighborhood $W\subset [0,1]$ of the 2-periodic orbit $\{p_1,p_2\}\subset [0,1]$ such
that if $x^{(0)} \in W$, then
$$\lim_{n\to\infty}(x^{(n)},y^{(n)})=\left\{\begin{array}{ll}
(p_1,p_1), \ \ \mbox{if} \ \ n=2m, \, m\to\infty\\[2mm]
(p_2,p_2), \ \ \mbox{if} \ \ n=2m+1, \, m\to\infty
\end{array}\right.$$
$$\left({\rm resp}. \ \ \lim_{n\to\infty}(x^{(n)},y^{(n)})=\left\{\begin{array}{ll}
(p_2,p_2), \ \ \mbox{if} \ \ n=2m, \, m\to\infty\\[2mm]
(p_1,p_1), \ \ \mbox{if} \ \ n=2m+1, \, m\to\infty
\end{array}\right.\right).$$
\end{itemize}
\end{thm}
\begin{proof} The proof follows from above mentioned results by applying Theorem 6.3 and Theorem 6.5 of \cite[p.216]{D}.
\end{proof}

\begin{rk} In part 4 of Theorem \ref{t2} the set $I$ (resp. $\gamma$) is the local unstable (resp. stable) manifold at $(x_*,x_*)$ (see \cite{D}).
   We note that the part 2 of this theorem gives a continuum set of 2-periodic orbits: any point of $[0,1]$ generates a 2-periodic orbit. But in part 5 we have only one 2-periodic orbit $\{(p_1,p_1),(p_2,p_2)\}$  of the operator (\ref{20}). This periodic orbit attracts other non-periodic trajectories.
\end{rk}

\section{Constrained algebras}

Let us first give some properties of the general algebra of bisexual population $\mathcal B$.

Extend the operator (\ref{3}) on $\R^{n+\nu}$, i.e. consider the
operator $V \colon \R^{n+\nu}\to \R^{n+\nu}$, $z=(x,y)\mapsto
V(z)=z'=(x',y')$ defined as
\begin{equation}\label{d1}
x'_j= \sum_{i,k=1}^{n,\nu}P_{ik,j}^{(f)}x_iy_k; \ \ y'_l=
\sum_{i,k=1}^{n,\nu} P_{ik,l}^{(m)}x_iy_k \, .
\end{equation}
Consider the following linear form $X \colon \R^n\to \R$ ( $Y \colon \R^\nu\to
\R$) defined by
\begin{equation}\label{d2}
 X(x)=\sum_{i=1}^nx_i, \ \ \left(Y(y)=\sum_{k=1}^\nu y_k\right) \, .
\end{equation}
Denote
\begin{equation}\label{d3}
 H_i=\{z=(x,y): X(x)=Y(y)=i\}, \ \ i=0,1 \, ,
\end{equation}
\begin{equation}\label{d33}
 Z_0=\{z=(x,y): X(x)Y(y)=0\}.
\end{equation}
the product of the $i$-hyperplanes in $\R^n$ and $\R^\nu$,
respectively, and the set of absolute nilpotent elements.

A subalgebra of an algebra is a vector subspace which is closed under the multiplication of vectors.
\begin{pro}\label{dp1}
\begin{itemize}
  \item[(1)] The sets $H_0$, $H_1$ and $S=S^{n-1}\times S^{\nu-1}$ are closed under the multiplication of the algebra $\mathcal B$.
  \item[(2)] There is one-to-one correspondence between one-dimensional subalgebras of $\mathcal B$ and the sets of non-zero fixed points of operator $V$ together with absolute nilpotent elements.
\end{itemize}
\end{pro}
\proof
(1) We give the proof for $H_1$ for other sets it is similar. Take any two point $z=(x,y),t=(u,v)\in H_1$ then by (\ref{5}) we obtain
$$
 X(zt) = \frac{1}{2} \sum_{k=1}^n\left(\sum_{i=1}^n\sum_{j=1}^\nu
P_{ij,k}^{(f)}(x_iv_j+u_iy_j)\right)= \frac{1}{2} \left(\sum_{i=1}^n\sum_{j=1}^\nu\left[
\sum_{k=1}^nP_{ij,k}^{(f)}\right](x_iv_j+u_iy_j)\right)$$
$$={1\over 2}(X(z)Y(t)+Y(z)X(t))=1.$$
Similarly we get $Y(zt)=1$, i.e. $zt\in H_1$.

(2) Let $\mathcal L_z=\R z$ be an one-dimensional subalgebra of $\mathcal B$ generated by $0\ne z\in \mathcal B.$ Consider $z^2=\lambda z$.

If $\lambda \neq 0$, then the element $p=\frac{z}{\lambda}$ is a fixed point of $V$ (since $p^2=p$).

If $\lambda =0$, then $z$ is an absolute nilpotent element of the algebra $\mathcal B$.

Now assume $z\ne 0$ is a fixed point of $V$, that is, $z^2=z$. Consider the space $\mathcal L_z=\R  z$, then for any $x,y\in \mathcal L_z$ we have
$$xy=(\lambda_1z)(\lambda_2z)=\lambda_1\lambda_2 z^2=\lambda_1\lambda_2z\in \mathcal L_z.$$
The same is true for an absolute nilpotent element.
This completes the proof.
\endproof

Let us study properties of the evolution algebra $\mathcal B_1$ with
constrains mentioned in the previous section.
\subsection{A hard constrained algebra.} Under hard constrain (\ref{v7}) the algebra
$\mathcal B_1$ has the following multiplication table
\begin{equation}\label{a1}
\begin{array}{llll}
e^{(f)}_ie^{(m)}_1 = e^{(m)}_1e^{(f)}_i=\frac{1}{2}
\left(e^{(f)}_i+e^{(m)}_1\right), \ \ i=1,\dots,n\\[3mm]
e^{(f)}_1e^{(m)}_k = e^{(m)}_ke^{(f)}_1=\frac{1}{2}
\left(e^{(f)}_1+e^{(m)}_k\right), \ \ k=1,\dots,\nu\\[3mm]
e^{(f)}_ie^{(m)}_k=e^{(m)}_ke^{(f)}_i=\frac{1}{2}
\left(e^{(f)}_1+e^{(m)}_1\right) \ \ \mbox{if} \ \ i\ne 1, k\ne 1\\[3mm]
e^{(f)}_ie^{(f)}_j=e^{(m)}_ke^{(m)}_l=0.
\end{array}
\end{equation}

\begin{pro} We have the following relations:
$$((e^{(m)}_1e^{(f)}_{i_1})e^{(f)}_{i_2})\dots e^{(f)}_{i_k}=\frac{1}{2^k}(e^{(m)}_1+e^{(f)}_{i_k}),$$
$$((e^{(m)}_1e^{(f)}_{1})e^{(m)}_{i_1})\dots e^{(m)}_{i_k}=\frac{1}{2^{k+1}}(e^{(f)}_1+e^{(m)}_{i_k}),$$

For $s\neq 1$ we have
$$((e^{(m)}_1e^{(f)}_{s})e^{(m)}_{i_1})\dots e^{(m)}_{i_k}=\frac{1}{2^{k+1}}(e^{(f)}_1+e^{(m)}_{i_k}), \ \ \mbox{if} \ \ \exists i_p\neq 1,\ \ p\in \{1,\dots,k\},$$
$$((e^{(m)}_1e^{(f)}_{s})e^{(m)}_{1})\dots e^{(m)}_{1}=\frac{1}{2^{k+1}}(e^{(f)}_s+e^{(m)}_{1}).$$
\end{pro}
\begin{proof} Straightforward.
\end{proof}

From this proposition we conclude that the algebra $\mathcal B_1$ is not nilpotent. Using products (\ref{a1}) we derive that $dim \mathcal B_1^2=n+\nu-1$ and $(\mathcal B_1^2)(\mathcal B_1^2)=\mathcal B_1^2$. Hence, the algebra $\mathcal B_1$ is not also solvable.

Now we shall investigate one-dimensional ideals of the algebra $\mathcal B_1$. From Proposition \ref{dp1} we conclude that if $I=\langle z=(x,y)\rangle$ is an ideal of the algebra $\mathcal B_1$, then $z\in H_0\cup H_1 \cup Z_0$.

We have

\begin{equation}\label{id1} ze_1^{(m)}=\frac{1}{2}\sum_{i=1}^nx_i(e^{(f)}_i+e^{(m)}_1)=
\frac{1}{2}\sum_{i=1}^nx_ie^{(f)}_i+\frac{1}{2}X(z)e^{(m)}_1,
\end{equation}
\begin{equation}\label{id2}
ze_1^{(f)}=\frac{1}{2}Y(z)e^{(f)}_1+ \frac{1}{2}\sum_{j=1}^\nu y_je^{(m)}_j.
\end{equation}
\begin{equation}\label{id3}
ze_k^{(m)}=\frac{1}{2}X(z)e_1^{(f)}+\frac{1}{2}x_1e_k^{(m)}+\frac{1}{2}(X(z)-x_1)e^{(m)}_1, \ \ k\geq 2,
\end{equation}
\begin{equation}\label{id4}
ze_k^{(f)}=\frac{1}{2}Y(z)e_1^{(m)}+\frac{1}{2}y_1e_k^{(f)}+\frac{1}{2}(Y(z)-y_1)e_1^{(f)}, \ \ k\geq 2.
\end{equation}

We consider the following three cases.

{\bf Case 1.} Let $z=(x,y)\in H_0.$ Then from equalities (\ref{id1}), (\ref{id2}) and a property of ideal we conclude that
$$ze_1^{(m)}=\frac{1}{2}\sum_{i=1}^nx_ie^{(f)}_i\in\langle z\rangle \Rightarrow z=x$$
$$ze_1^{(f)}=\frac{1}{2}\sum_{j=1}^\nu y_je^{(m)}_j\in\langle z\rangle \Rightarrow z=y.
$$
Therefore, $z=0$.

{\bf Case 2.} Let $z=(x,y)\in H_1.$ Then from equalities (\ref{id1}), (\ref{id2}) and the property of ideal we conclude that
$$ze_1^{(m)}=\frac{1}{2}\sum_{i=1}^nx_ie^{(f)}_i+\frac{1}{2}e^{(m)}_1\in\langle z\rangle \Rightarrow z=x+y_1e^{(m)}_1,$$
$$ze_1^{(f)}=\frac{1}{2}e^{(f)}_1+ \frac{1}{2}\sum_{j=1}^\nu y_je^{(m)}_j\in\langle z\rangle \Rightarrow z=x_1e^{(f)}_1+y.
$$

Consequently, $z=e^{(f)}_1+e^{(m)}_1.$

{\bf Case 3.} Let $z=(x,y)\in Z_0.$

{\bf Subcase 3.1.} Let $X(z)=0$ and $Y(z)\neq 0$. Then from the equality (\ref{id1}) we conclude
that
$$ze_1^{(m)}=\frac{1}{2}\sum_{i=1}^nx_ie^{(f)}_i\in\langle z\rangle \Rightarrow z=x.$$

Thus we have obtained that $z=x$, but this is a contradiction with condition $Y(y)\ne 0.$

{\bf Subcase 3.2.} Let $X(z)\neq0$ and $Y(z)=0$. Then from the equality (\ref{id2}) we get that
$$ze_1^{(f)}=\frac{1}{2}\sum_{j=1}^\nu y_je^{(m)}_j \in\langle z\rangle \Rightarrow x_i=0, \ 1\leq i \leq n,$$
which is a contradiction with condition $X(z)\neq0$.

Thus we have proved the following

\begin{pro} The ideal $\langle e^{(f)}_1+e^{(m)}_1\rangle$ is unique (up to isomorphism) one-dimensional ideal of the algebra $\mathcal B_1.$
\end{pro}

Therefore, the algebra $\mathcal B_1$ is not simple (in sense that it has non-trivial ideal).

\begin{hyp} \label{hyp} In the algebra $\mathcal B_1$ there is not ideals except one-dimensional $\langle e^{(f)}_1+e^{(m)}_1\rangle.$
\end{hyp}

Below we present several cases when the hypothesis is true.

For $s\geq 2$ let $I=\langle z_1, z_2, \dots, z_s\rangle$ be $s$-dimensional ideal of the algebra $\mathcal B_1$, where $z_t=\displaystyle\sum_{i=1}^{n}x_{i,t}e_i^{(f)}+ \sum_{j=1}^{\nu}y_{j,t}e_j^{(m)}$ with $1\leq t \leq s$.

From the equalities (\ref{id1}) - (\ref{id4}) we have
$$z_te_1^{(m)}=\frac{1}{2}\sum_{i=1}^nx_{i,t}e^{(f)}_i+\frac{1}{2}X(z_t)e^{(m)}_1\in I,$$
$$z_te_1^{(f)}=\frac{1}{2}Y(z_t)e^{(f)}_1+ \frac{1}{2}\sum_{j=1}^\nu y_{j,t}e^{(m)}_j\in I,$$
$$z_te_k^{(m)}=\frac{1}{2}X(z_t)e_1^{(f)}+\frac{1}{2}x_{1,t}e_k^{(m)}+\frac{1}{2}(X(z_t)-x_{1,t})e^{(m)}_1\in I,$$
$$z_te_k^{(f)}=\frac{1}{2}Y(z_t)e_1^{(m)}+\frac{1}{2}y_{1,t}e_k^{(f)}+\frac{1}{2}(Y(z_t)-y_{1,t})e_1^{(f)}\in I.$$

Let $z$ be an arbitrary element of the ideal $I$. Then $z=\displaystyle\sum_{k=1}^{s}\alpha_kz_k.$

From the first of the above-mentioned equalities we obtain
$$\frac{1}{2}\sum_{i=1}^nx_{i,t}e^{(f)}_i+\frac{1}{2}X(z_t)e^{(m)}_1=\sum_{k=1}^{s}\alpha_{k,t}z_k=
\sum_{k=1}^{s}\alpha_{k,t}\sum_{i=1}^nx_{i,k}e^{(f)}_i+\sum_{k=1}^{s}\alpha_{k,t}\sum_{j=1}^{\nu} y_{j,k}e^{(m)}_j \Rightarrow $$

\begin{equation}\label{s-id11}
x_{i,t}=2\sum_{k=1}^{s}\alpha_{k,t}x_{i,k}, \quad 1\leq i \leq n \quad
X(z_t)=2\sum_{k=1}^{s}\alpha_{k,t} y_{1,k}, \quad \sum_{k=1}^{s}\alpha_{k,t} y_{j,k}=0, \quad 2\leq j \leq \nu.
\end{equation}

From the second of the above equalities we have
$$\frac{1}{2}Y(z_t)e^{(f)}_1+ \frac{1}{2}\sum_{j=1}^\nu y_{j,t}e^{(m)}_j=\sum_{k=1}^{s}\beta_{k,t}z_k=
\sum_{k=1}^{s}\beta_{k,t}\sum_{i=1}^nx_{i,k}e^{(f)}_i+\sum_{k=1}^{s}\beta_{k,t}\sum_{j=1}^{\nu} y_{j,k}e^{(m)}_j \Rightarrow $$
\begin{equation}\label{s-id22}
Y(z_t)=2\sum_{k=1}^{s}\beta_{k,t}x_{1,k}, \quad \sum_{k=1}^{s}\beta_{k,t}x_{i,k}=0, \ 2\leq i \leq n, \quad
y_{j,t}=2\sum_{k=1}^{s}\beta_{k,t} y_{j,k}, \ 1\leq j \leq \nu.
\end{equation}

From the third of the above equalities we have
$$\frac{1}{2}X(z_t)e_1^{(f)}+\frac{1}{2}x_{1,t}e_k^{(m)}+\frac{1}{2}(X(z_t)-x_{1,t})e^{(m)}_1=\sum_{l=1}^{s}\gamma_{l,t}\sum_{i=1}^nx_{i,l}e^{(f)}_i+\sum_{l=1}^{s}\gamma_{l,t}\sum_{j=1}^{\nu} y_{j,l}e^{(m)}_j\Rightarrow$$
\begin{equation}\label{s-id3.11}
X(z_t)=2\sum_{l=1}^{s}\gamma_{l,t}x_{1,l}, \quad \sum_{l=1}^{s}\gamma_{l,t}x_{i,l}=0, \quad 2\leq i \leq n,
\end{equation}
\begin{equation}\label{s-id3.22}
X(z_t)-x_{1,t}=2\sum_{l=1}^{s}\gamma_{l,t} y_{1,l}, \quad x_{1,t}=2\sum_{l=1}^{s}\gamma_{l,t} y_{k,l}, \quad \sum_{l=1}^{s}\gamma_{l,t} y_{j,l}=0, \quad j\notin \{1, k\}.
\end{equation}

From the fourth of the above equalities we derive
$$\frac{1}{2}Y(z_t)e_1^{(m)}+\frac{1}{2}y_{1,t}e_k^{(f)}+\frac{1}{2}(Y(z_t)-y_{1,t})e_1^{(f)}=\sum_{l=1}^{s}\delta_{l,t}\sum_{i=1}^nx_{i,l}e^{(f)}_i+\sum_{l=1}^{s}\delta_{l,t}\sum_{j=1}^{\nu} y_{j,l}e^{(m)}_j\Rightarrow$$

\begin{equation}\label{s-id4.11}
\frac{1}{2}Y(z_t)=\sum_{l=1}^{s}\delta_{l,t}y_{1,l}, \quad \sum_{l=1}^{s}\delta_{l,t}y_{j,l}=0, \quad 2 \leq j\leq  m,
\end{equation}
\begin{equation}\label{s-id4.22}
Y(z_t)-y_{1,t}=2\sum_{l=1}^{s}\delta_{l,t}x_{i,l}, \quad \frac{1}{2}y_{1,t}=\sum_{l=1}^{s}\delta_{l,t}x_{k,l}, \quad
\sum_{l=1}^{s}\delta_{l,t}x_{i,l}=0, \quad i\notin\{1,k\}.
\end{equation}

Consider the matrices $A=(2\alpha_{i,j})_{i,j=1, \dots, s},$ $B=(2\beta_{i,j})_{i,j=1, \dots, s},$ $C=(2\gamma_{i,j})_{i,j=1, \dots, s}$ and $D=(2\delta_{i,j})_{i,j=1, \dots, s}.$

Let $X=(X(z_1), \dots, X(z_s)),$ (respectively, $Y=(Y(z_1), \dots, Y(z_s))$).
Summing over $i$ the first equality of (\ref{s-id11}) we get $A^tX^t=X^t$. Similarly using (\ref{s-id11}) - (\ref{s-id22}) we obtain
$B^tY^t=Y^t, \ A^tY^t=X^t, \ B^tX^t=Y^t$.

{\bf Case 1.} Let $det A \cdot det B\neq 0.$ Then from
$$\sum_{k=1}^{s}\alpha_{k,t} y_{j,k}=0, \quad 2\leq j \leq \nu \quad \mbox{and} \ \sum_{k=1}^{s}\beta_{k,t}x_{i,k}=0, \ 2\leq i \leq n$$
we derive that $x_{i,k}=y_{j,k}=0$ for any $i,j\neq 1.$ Moreover, from
$A^tX^t=X^t, \ A^tY^t=X^t$ we deduce
$A^t(X^t-Y^t)=0 \Rightarrow X^t=Y^t.$ Therefore, $x_{1,t}=X(z_t)=Y(z_t)=y_{1,t}$ which mean that $z_t=x_{1,t}(e_1^{(f)}+e_1^{(m)})$ for any $t\in \{1, \dots, s\} \Rightarrow $ ideal $I$ has dimension one.

{\bf Case 2.} Let $det A\neq 0$ and  $det B=0$. Then from
$$\sum_{k=1}^{s}\alpha_{k,t} y_{j,k}=0, \quad 2\leq j \leq \nu$$
we have $y_{j,k}=0$ for any $j\neq 1$ and hence, $Y(z_k)=y_{1,k}$. Moreover, from
$A^tX^t=X^t, \ A^tY^t=X^t$ we deduce $A^t(X^t-Y^t)=0 \Rightarrow X^t=Y^t.$ Therefore,
$$X(z_t)=Y(z_t)=y_{1,t}.$$

Consider
$$x_{1,t}=2\sum_{l=1}^{s}\gamma_{l,t} y_{k,l}, \ \ 2\leq k \Rightarrow x_{1,t}=0,\ \ 1 \leq t \leq s. $$
$$Y(z_t)=2\sum_{k=1}^{s}\beta_{k,t}x_{1,k} \Rightarrow y_{1,t}=0.$$

So, we have that $z_t=\displaystyle\sum_{i=2}^{n}x_{i,t}e_i^{(f)}.$

The equation $$\sum_{k=1}^{s}\beta_{k,t}x_{i,k}=0, \ 2\leq i \leq n$$ derive that $\langle z_1, \dots, z_s\rangle$ are not linear independent system, i.e., $dim I<s$ which is a contradiction.

{\bf Case 3.} Let $det A=0$ and  $det B\neq0$. This case also lead to contradiction similar to Case 2.

{\bf Case 4.} Let $det A=det B=0$.

{\bf Case 4.1} Let $det G\neq 0$. Then from
$\displaystyle\sum_{l=1}^{s}\gamma_{l,t}x_{i,l}=0, \quad 2\leq i \leq n$ we have $x_{i,1}=x_{i,2}=x_{i,3}=\dots=x_{i,s}=0, \quad 2\leq i \leq n.$

From
$$X(z_t)-x_{1,t}=2\sum_{l=1}^{s}\gamma_{l,t} y_{1,l}$$
we conclude that $$2\sum_{l=1}^{s}\gamma_{l,t} y_{1,l}=0 \Rightarrow y_{1,1}=y_{1,2}=y_{1,3}=\dots=y_{1,s}=0.$$

Therefore, from
$$X(z_t)=2\sum_{k=1}^{s}\alpha_{k,t} y_{1,k}$$ we have that $X(z_t)=x_{i,t}=0, 1\leq i \leq n.$

From
$$x_{1,t}=2\sum_{l=1}^{s}\gamma_{l,t} y_{k,l}, \quad \sum_{l=1}^{s}\gamma_{l,t} y_{j,l}=0, \quad j\notin \{1, k\}$$
we conclude that  $y_{i,1}=y_{i,2}=\dots=y_{i,s}=0$ for $2\leq i \leq \nu.$ Thus we obtain that $z_t=0, 1 \leq t \leq s$ which is a contradiction.

{\bf Case 4.2} Let $det D \neq 0$. The applying similar arguments as in previous subcase we get a contradiction.

In order to check correctness of Hypothesis it is necessary to consider the last case when $det A=det B=det C=det D=0.$ Hence, we reduced Hypothesis \ref{hyp} to consideration of the case $det A=det B=det C=det D=0.$

\subsection{Quotient algebra.}

Let $I=\langle e_1^{(f)}+e_1^{(m)}\rangle$. Consider quotient algebra $\overline{B_1}=B_1/I$ with
the basis $\{\overline{e_1^{(f)}}, \overline{e_2^{(f)}}, \dots, \overline{e_n^{(f)}}, \overline{e_2^{(m)}}, \dots, \overline{e_\nu^{(m)}}\}.$ Then the table of multiplication of the algebra $\overline{B_1}$ is as follows:

\begin{equation}\label{a2}
\begin{array}{llll}
\overline{e^{(f)}_i}\cdot\overline{e^{(f)}_1} = \overline{e^{(f)}_1}\cdot\overline{e^{(f)}_i}=\frac{1}{2}
\left(\overline{e^{(f)}_1}-\overline{e^{(f)}_i}\right), \ \ i=2,\dots,n\\[3mm]
\overline{e^{(m)}_k}\cdot\overline{e^{(f)}_1}=\overline{e^{(f)}_1}\cdot\overline{e^{(m)}_k} = \frac{1}{2}
\left(\overline{e^{(f)}_1}+\overline{e^{(m)}_k}\right), \ \ k=2,\dots,\nu,
\end{array}
\end{equation}
where the omitted products are equal to zero.

\begin{rk} Note that $B_1$ and $I$ are symmetric with respect to exchange of basis of "females" with "males". 
Such a symmetry exists in  $\overline{B_1}=B_1/I$ meaning that
$\overline{B_1}$ can be given by two (symmetric) ways

i) with the basis $$\{\overline{e_1^{(f)}}, \overline{e_2^{(f)}}, \dots, \overline{e_n^{(f)}}; 
\overline{e_2^{(m)}}, \overline{e_3^{(m)}}, \dots, \overline{e_\nu^{(m)}}\}$$ 
and the multiplication table given by (\ref{a2});

ii)  with the basis $$\{\overline{e_2^{(f)}}, \dots, \overline{e_n^{(f)}};
\overline{e_1^{(m)}}, \overline{e_2^{(m)}}, \dots, \overline{e_\nu^{(m)}}\}$$ and the multiplication table 
$$
\overline{e_k^{(m)}}\cdot\overline{e_1^{(m)}}=\overline{e_1^{(m)}}\cdot\overline{e_k^{(m)}}={1\over 2}\left(\overline{e_1^{(m)}}-\overline{e_k^{(m)}}\right), \, k=2,\dots,\nu;$$
$$\overline{e_i^{(f)}}\cdot\overline{e_1^{(m)}}=\overline{e_1^{(m)}}\cdot\overline{e_i^{(f)}}={1\over 2}\left(\overline{e_1^{(m)}}+\overline{e_i^{(f)}}\right), \, i=2,\dots,n.$$ 
\end{rk}
The following theorem characterizes ideals of $\overline{B_1}$.

\begin{thm} An arbitrary ideal $J$ of $\overline{B_1}$ is either
$$J=\langle\overline{e^{(f)}_1}-\overline{e^{(f)}_i}, \overline{e^{(f)}_1}+\overline{e^{(m)}_j}, \ 2\leq i,j\rangle$$
or $$J=\langle\sum_{i=2}^{n}a_i\overline{e^{(f)}_i}, \ \sum_{j=2}^{\nu}b_j\overline{e^{(m)}_j}, \ \sum_{i=2}^{n}a_i=\sum_{j=2}^{\nu}b_j=0\rangle.$$
\end{thm}
\begin{proof}
Let  $J$ be an ideal of $\overline{B_1}$ and $z=\displaystyle\sum_{i=1}^{n}a_i\overline{e^{(f)}_i}+\sum_{j=2}^{\nu}b_j\overline{e^{(m)}_j}$ be an arbitrary element of $J$.

If $\overline{e^{(f)}_1}\in J$ then by (\ref{a2}) we get $J=\overline{B_1}$. Therefore, we shall assume that $\overline{e^{(f)}_1}\notin J.$

Consider
$$z_1=2z\cdot\overline{e^{(f)}_1}=(\sum_{j=2}^{n}b_j+\sum_{i=2}^{n}a_i)\overline{e^{(f)}_1}-
\sum_{i=2}^{n}a_i\overline{e^{(f)}_i}+\sum_{j=2}^{\nu}b_j\overline{e^{(m)}_j},$$
$$z_2=z_1-z=(\sum_{j=2}^{n}b_j+\sum_{i=2}^{n}a_i-a_1)\overline{e^{(f)}_1}-2\sum_{i=2}^{n}a_i\overline{e^{(f)}_i},$$
$$z_3=z_1+z=(\sum_{j=2}^{n}b_j+\sum_{i=2}^{n}a_i+a_1)\overline{e^{(f)}_1}+2\sum_{j=2}^{\nu}b_j\overline{e^{(m)}_j},$$
$$z_4=z_2\cdot \overline{e^{(f)}_1}=\sum_{i=2}^{n}a_i\overline{e^{(f)}_i}-
(\sum_{i=2}^{n}a_i)\overline{e^{(f)}_1},$$
$$z_5=z_3\cdot\overline{e^{(f)}_1}=
\sum_{j=2}^{\nu}b_j\overline{e^{(m)}_j}+(\sum_{j=2}^{\nu}b_j)\overline{e^{(f)}_1},$$
$$z_6=z_2+2z_4=(\sum_{j=2}^{n}b_j-\sum_{i=2}^{n}a_i-a_1)\overline{e^{(f)}_1}.$$
Since $J$ is an ideal we should have $z_i\in J$, $i=1,\dots,6$. Consequently, since $\overline{e^{(f)}_1}\notin J$ we obtain
$$\sum_{j=2}^{n}b_j-\sum_{i=2}^{n}a_i=a_1.$$

{\bf Case 1.} Let $a_1\neq 0$. We put
$$u_i=z\cdot \frac{\overline{e^{(f)}_i}}{a_1}=(\overline{e^{(f)}_1}-\overline{e^{(f)}_i}),$$
$$v_j=z\cdot \frac{\overline{e^{(f)}_m}}{a_1}=(\overline{e^{(f)}_1}+\overline{e^{(m)}_j}).$$

For $2\leq s \leq n$ we should have
$$z+\sum_{i=2, i\neq s}^{n}a_iu_i-\sum_{j=2}^{\nu}b_jv_j=-a_s(\overline{e^{(f)}_1}-\overline{e^{(f)}_s})\in J.
$$

Similarly,
$$b_t(\overline{e^{(f)}_1}+\overline{e^{(m)}_t})\in J, \ 2\leq t \leq \nu.$$

If $z\neq 0,$ then either $\overline{e^{(f)}_1}-\overline{e^{(f)}_s} \in J$ for some $s$ or
$\overline{e^{(f)}_1}+\overline{e^{(m)}_t}$ for some $t$.

At any rate, by multiplication of these elements to $\overline{e^{(f)}_i}, \ 2\leq i \leq n$ and
$\overline{e^{(m)}_j}, \ 2\leq j \leq \nu$ we obtain that $\overline{e^{(f)}_1}-\overline{e^{(f)}_i}, \ \overline{e^{(f)}_1}+\overline{e^{(m)}_j}\in J$ for any $2\leq i,j.$

{\bf Case 2.} Let $a_1=0$. Then applying the same arguments as in Case 1 for the elements
$z_2, \ z_4, z_5$ we conclude that
$\displaystyle\sum_{j=2}^{n}b_j=\sum_{i=2}^{n}a_i=0 \Rightarrow \sum_{i=2}^{n}a_i\overline{e^{(f)}_i}\in J$ and $\displaystyle\sum_{j=2}^{\nu}b_j\overline{e^{(m)}_j}\in J.$
\end{proof}

\subsection{Case $n=\nu=2$} In this case using (\ref{abc}) the corresponding four dimensional evolution algebra
has the following multiplication table
  \begin{equation}\label{a22}
\begin{array}{llll}
e^{(f)}_1e^{(m)}_1 = e^{(m)}_1e^{(f)}_1=\frac{1}{2}
\left((1-a)e^{(f)}_1+ae^{(f)}_2+(1-c)e^{(m)}_1+ce^{(m)}_2\right),\\[3mm]
e^{(f)}_2e^{(m)}_1 = e^{(m)}_1e^{(f)}_2=\frac{1}{2}
\left((1-b)e^{(f)}_1+be^{(f)}_2+e^{(m)}_1\right),\\[3mm]
e^{(f)}_1e^{(m)}_2 = e^{(m)}_2e^{(f)}_1=\frac{1}{2}
\left(e^{(f)}_1+(1-d)e^{(m)}_1+de^{(m)}_2\right),\\[3mm]
e^{(f)}_2e^{(m)}_2 = e^{(m)}_2e^{(f)}_2=\frac{1}{2}
\left(e^{(f)}_1+e^{(m)}_1\right),\\[3mm]
e^{(f)}_ie^{(f)}_j=e^{(m)}_ke^{(m)}_l=0.
\end{array}
\end{equation}

Let us consider the following change of basis (which preserves sexual distinction)

$${e_1^{(f)}}'=A_1e_1^{(f)}+A_2e_2^{(f)}, \quad {e_2^{(f)}}'=B_1e_1^{(f)}+B_2e_2^{(f)}$$
$${e_1^{(m)}}'=C_1e_1^{(m)}+C_2e_2^{(m)}, \quad {e_2^{(m)}}'=D_1e_1^{(m)}+D_2e_2^{(m)}$$
with $(A_1B_2-A_2B_1)(C_1D_2-C_2D_1)\neq 0.$

\begin{thm}\label{t4} The table of multiplications (\ref{a22}) preserves in new basis $\{{e_1^{(f)}}'$, ${e_2^{(f)}}'$,
${e_1^{(m)}}'$, ${e_2^{(m)}}'\}$ if the following relations hold true:
$$a'=(A_1B_2-A_2B_1)^{-1}(A_1+A_2)(cA_1C_1-A_2C_1+dA_2C_1-A_2C_2),$$
$$b'=(A_1B_2-A_2B_1)^{-1}(cA_1B_1C_1-A_2B_1C_1+cA_2B_1C_1$$
$$+dA_1B_2C_1-A_2B_2C_1+dA_2B_2C_1-A_2B_1C_2-A_2B_2C_2)$$
$$c'=(C_1D_2-C_2D_1)^{-1}(C_1+C_2)(aA_1C_1-A_1C_2+bA_1C_2-A_2C_2),$$
$$d'=(C_1D_2-C_2D_1)^{-1}(aA_1C_1D_1-A_1C_2D_1+aA_1C_2D_1$$
$$-A_2C_2D_1+bA_1C_1D_2-A_1C_2D_2+bA_1C_2D_2-A_2C_2D_2).$$
\end{thm}
\begin{proof} The proof is carried out by straightforward calculations.
\end{proof}
This theorem is useful to reduce the number of parameters in an algebra.
For example, in case $A_1+A_2=0$ and $C_1+C_2=0$ we get $a'=c'=0$ and
the corresponding algebra remains with two parameters $b'$ and $d'$.
Since male and female are symmetric Theorem \ref{t4} is true (up to rename parameters) when
the change of basis will exchange male basis elements with female basis elements.

\section{Biological interpretation of the results.}

In biology, a population genetics is the study of
the distributions and changes of allele (type) frequency in a population.

The results formulated in previous sections  have the following biological interpretations:

Let $z^{(0)}=(x^{(0)},y^{(0)})\in \mathcal S$ be an initial state, i.e., $x^{(0)}$ (resp. $y^{(0)}$)
is the probability distribution on the set
$\{1,\dots, n\}$ (resp. $\{1,\dots,\nu\}$) of genotypes.
Assume the trajectory $z^{(m)}$ of this point has a limit
$z=((x_1,\dots,x_n),(y_1,\dots,y_\nu))$ this means that
the future of the population is stable: female genotypes $i$ survives with probability $x_i$ and male genotype
$j$ survives with probability $y_j$. A genotype will disappear if its probability is zero.\\

{\it The case of hard constrain:}

\begin{itemize}
\item[(a)] The population has a continuum set of equilibrium states (Proposition \ref{p1}), it
stays in a neighborhood of one of the equilibrium states of the
population (stable fixed point).

\item[(b)] If initially the type $1$ of females {\it and} males have a positive {\it equal} probability, then
in the future of the evolution the type $1$ of both sexes will survive with probability 1 and consequently all
other types will disappear. If initially the type $1$ of females (resp. males) have positive probability which is strictly large
   than initial probability of the type $1$ of males (resp. females) then type $1$ of females (resp. males) survives with probability 1, but all types of males (resp. females) survive with probabilities {\it less} than 1 (interpretation of the part (i) of Theorem \ref{ts}).

\item[(c)] If initially the type $1$  of females is not present (i.e., has probability zero) but the type 1 of
males has a positive probability, say $y^{(0)}_1$, then
in the future of the evolution the type $1$ of females
survives with probability $1-y^{(0)}_1$, moreover another types, say $i$, of
females survives if and only if it was present initially (i.e. $x^{(0)}_i>0$)
the probability of its survive is $y_1^{(0)}x^{(0)}_i$;
for males only type 1 survives with probability 1 ((ii) of Theorem \ref{ts}).
Moreover, the population goes to its equilibrium starting from the first generation.

Similar interpretation also can be made for the part (iii).

We note that from the parts (ii) and (iii) of Theorem \ref{ts} it
follows that if initially the type 1 of both sexes are not present
then still this type 1 will survive with probability 1 (as in the case (b)).
 Moreover, this initial state, say $((0,x^{(0)}_2,\dots,x^{(0)}_n),(0,y^{(0)}_2,\dots,y^{(0)}_\nu))$,
 goes to the equilibrium state $((1,0,\dots,0)$, $(1,0,\dots,0))$ in the next (first) generation.
 This is a consequence of the hard constrain which is assumed for the population.\\
\end{itemize}

{\it The case of} $n=\nu=2$:

\begin{itemize}

\item[(d)] Depending on the parameters the population my have one, two or infinitely many (continuum) equilibrium states (Proposition \ref{p2}).
 The population stays in a neighborhood of one of the equilibrium states of the
population (stable fixed point).

\item[(e)] In case $0\leq a=b=c=d<1$ the types of the population will survive. The future
of the population is stable which has equilibrium
$\left(({1\over 1+a},{a\over 1+a}), ({1\over 1+a},{a\over 1+a})\right)$ (part 1, of Theorem \ref{t2}).

\item[(f)] In case $a=b=c=d=1$ the population is periodic, i.e., the probability of a type will 2-periodically increase and decrease. For example,
 if the probability of the type 1, (i.e.,  $1-x^{(0)}$) is less then ${1\over 2}$ the in the next generation the probability will be large than ${1\over 2}$, after that it will be less than ${1\over 2}$ and so on. Moreover, the population goes to this periodic state starting at the first generation (part 2, of Theorem \ref{t2}).

\item[(g)] Under conditions of the part 3 of Theorem \ref{t2} all types of the population will survive if the initial state
of the population is sufficiently close to the equilibrium state $((1-x^*,x^*),(1-x^*,x^*))$ and in the future the state of population
will be in the neighborhood of this equilibrium state.

\item[(h)] Under conditions of the part 4 of Theorem \ref{t2} there are some initial states of the
population for which in the future it will stay in a neighborhood
of the equilibrium state $((1-x^*,x^*),(1-x^*,x^*))$. While there are other initial states for which the trajectory
of the states will go out from a neighborhood of the equilibrium.

\item[(i)] Under conditions of part 5 of Theorem \ref{t2}, depending on the initial state the population
may go to the 2-periodic state.
 \end{itemize}

\begin{rk} 1) Note that in biology there are some species where population have reasonably predictable patterns of change
 although the full reasons for population periodicity is one of the major unsolved ecological problems.
 There are a number of factors which influence population change such as availability of food, predators,
 diseases and climate. The periodicity mentioned in items (f) and (i) can be examples of such biological phenomenon.

 2) It is known Bernstein's problem \cite{ly}, \cite{ly1}, which is related to a fundamental statement of population
genetics, the so-called stationarity principle. This principle holds provided that the
Mendel law is assumed, but it is consistent with other mechanisms of heredity. An
adequate mathematical problem is as follows. A quadratic stochastic operator $V$ is a
Bernstein mapping if $V^2 = V$. This property is just the stationarity principle. This property is known as
Hardy-Weinberg law \cite{HS}. The problem is to describe all Bernstein mappings explicitly.

We note that the equality $V^2=V$ means that for any $z\in S$ the point $V(z)$ is a fixed point for $V$.
As a corollary of Theorems \ref{ts} and \ref{t2} we can say that the operators (\ref{v3a}) and (\ref{22})
are not Bernstein mapping. But for the operator (\ref{v3a}) the condition $H^2(z)=H(z)$ is satisfied for any
$z=(x,y)\in S$ with $x_1y_1=0$.
\end{rk}

\section*{ Acknowledgements}
U.Rozikov thanks Aix-Marseille University Institute for Advanced Study IM\'eRA
(Marseille, France) for support by a residency scheme. The work also partially
supported by the Grant No. 0828/GF4 of Ministry of Education and Science of the Republic of Kazakhstan. 
We thank the referee for helpful suggestions.

{}
\end{document}